\documentclass[11pt,reqno]{amsart}
\usepackage{amssymb}
\usepackage{amsthm}
\usepackage{eucal}
\usepackage{a4wide}

\newcommand{\R}{\mathbb R}

\newcommand{\T}{\mathbb T}

\newcommand{\p}{\partial}

\newcommand{\secao}[1]{\section{#1}\setcounter{equation}{0}}

\newtheorem{theorem}{Theorem}[section]
\newtheorem{proposition}[theorem]{Proposition}
\newtheorem{remark}[theorem]{Remark}
\newtheorem{lemma}[theorem]{Lemma}

\begin{document}
\title[Global well-posedness]{Global well-posedness for a coupled modified KdV system}
\author{Ad\'an J. Corcho}
\address{Instituto de Matem\'atica, Universidade Federal do Rio de Janeiro-UFRJ. Ilha do Fund\~{a}o,
21945-970. Rio de Janeiro-RJ, Brazil}
\email{adan.corcho@pq.cnpq.br}
\thanks{A. J. Corcho was supported by CAPES and CNPq, Brazil.}
\author{Mahendra Panthee}
\address{Centro de Matem\'atica, Universidade do Minho, 4710-057, Braga, Portugal}
\email{mpanthee@math.uminho.pt}
\thanks{M. Panthee was partially supported by FCT/CAPES project and  the Research Center of Mathematics of the University of Minho, Portugal through the FCT Pluriannual Funding Program and through the project PTDC/MAT/109844/2009.}


\begin{abstract}
We prove the sharp global well-posedness results for the initial value problems (IVPs) associated to the modified Korteweg-de Vries (mKdV) equation and a system modeled by the coupled modified Korteweg-de Vries equations (mKdV-system). To obtain our results we use the second generation of the modified energy and almost conserved quantities, more precisely, the refined $I$-method introduced by Colliander, Keel, Staffilani, Takaoka and Tao in \cite{CKSTT, CKSTT-2}.
\end{abstract}

\maketitle

\secao{Introduction}
We consider the initial value problems (IVPs) associated to the modified Korteweg-de Vries (mKdV) equation
\begin{equation}\label{ivpmkdv}
\p_tu + \p_x^3u +\p_x(u^3)=0,\quad u(x,0)=\phi(x),
\end{equation}
and system of the mKdV equations
\begin{equation}\label{ivp-sy}
\begin{cases}
\p_tu + \p_xu + \p_x(uv^2) =0,&u(x,0)=\phi(x),\\
\p_tv + \p_x^3v + \p_x(u^2v) =0,& v(x,0)=\psi(x),\\
\end{cases}
\end{equation}
where $(x, t) \in \R\times \R$; $u= u(x, t)$ and $v= v(x, t)$ are real-valued functions.

\subsection{Brief review about well-posedness} An extensive study of the  IVP \eqref{ivpmkdv} can be found in the literature, see for example the works \cite{KPV7, KPV3, FLP1} and references therein. The mKdV equation is a completely integrable model and has also been studied in the inverse scattering theory, \cite{Miura, Scott}. The system \eqref{ivp-sy}  contains a pair of mKdV type equations
coupled through nonlinear parts and is a special case of a broad class of nonlinear evolution equations considered by Ablowiz, Kaup, Newell and Segur \cite{AKNS} in the inverse scattering context.

Kenig, Ponce and Vega \cite{KPV3} proved that the IVP \eqref{ivpmkdv} is locally well-posed for given data in $ H^s(\R)$,
 $s\geq \frac{1}{4}$.  To obtain this result, they used the sharp version of the smoothing effects of Kato type (see \cite{Kato1}) satisfied by the group associated to the linear problem combined with the contraction
 mapping principle. This local result is sharp. Note that the conservation laws
 \begin{equation}\label{con.1}
M(t):= \frac12 \|u(t)\|_{L^2(\R)}^2, \qquad (Mass)
\end{equation}
\begin{equation}\label{con.2}
E(t):= \frac12 \|u_x(t)\|_{L^2(\R)}^2-\frac1{12}\|u(t)\|_{L^4(\R)}^4, \qquad (Energy)
\end{equation}
satisfied by the mKdV flow permit to extend the local solution to the global one in $H^s(\R)$, $s\geq 1$.

Following the similar argument used in \cite{KPV3}, Montenegro (see \cite{Monte}) proved that the IVP (\ref{ivp-sy}) is locally
well-posed for given data $(\phi, \psi)$ in $H^s(\R)\times H^s(\R)$, $s\geq \frac{1}{4}$. Moreover, using the
conservation laws
\begin{equation}\label{con21}
I_1(u, v) := \int_{\R} (u^2 +v^2)\,dx
\end{equation}
and
\begin{equation}\label{con22}
I_2(u, v) := \int_{\R} (u_x^2 + v_x^2 - u^2v^2)\, dx,
\end{equation}
satisfied by the flow of (\ref{ivp-sy}), the local solution can be extended to a global one for given data in $H^s(\R)\times H^s(\R)$, $s\geq 1$.

 The system \eqref{ivp-sy} has also been studied from the point of view of the abstract stability theory of Gri\-llakis, Shatah and Strauss developed in  \cite{GSS1}. Using this theory,  Montenegro  (see \cite{Monte}) proved the orbital stability of solitary wave solutions to the IVP (\ref{ivpmkdv}). Recently, Alarcon, Angulo and Montenegro (see \cite{AGM1}) considered a general class of nonlinear dispersive system containing the IVP (\ref{ivpmkdv}) and proved existence, orbital stability and nonlinear instability of solitary wave solutions. To get existence and stability results the  used the concentration-compactness method and to get nonlinear instability they followed a method established by Bona, Souganidis and Strauss in \cite{BSS1} to analyze the instability of solitary waves
 of the KdV type equations.

As discussed above, for the both IVPs \eqref{ivpmkdv} and \eqref{ivp-sy}, there is a gap in Sobolev indices to get global solution for which local solution already exists. In the range  of the Sobolev indices $\frac14 \leq s <1$, one cannot obtain an {\em a priori} estimate to prove global well-posedness with the usual iteration process. There are several attempts to overcome this difficulty so as to get obtain global well-posedness for the IVP \eqref{ivpmkdv} for $s\geq \frac14$.

A pioneer technique  to get the global solution for given data below energy spaces was introduced by Bourgain in \cite{B20}. This technique consists of splitting the given  data in low and high frequency parts and resolving auxiliary IVPs with new sets of data there by  creating an iteration process in the energy space to extend  the local in time solution to the global one. Several authors have  applied this technique to obtain the global solution to various nonlinear dispersive models.  Fonseca, Linares and
Ponce (see \cite{FLP1}) simplified this technique to get the global solution to the mKdV equation in $H^s(\R)$, $s > \frac{3}{5}$.  It is also applied to get the global solutions to the semi-linear wave equations  (see \cite{KPV8}) and critical generalized KdV equations (see \cite{FLP2}).  Also, Takaoka used this technique to get the global solutions to KP-II equation in
\cite{TAKA2} and to the Schr\"{o}dinger equation with derivative in \cite{TAKA}.  Further, Pecher (see \cite{HP}) followed the same technique to prove the global well-posedness for the 1D Zakharov system below the energy space. Recently, using the argument in \cite{FLP1}, Carvajal  (see \cite{CL10}) proved that the IVP associated to the higher order nonlinear Schr\"odinger equation is globally well-posed in $H^s(\R)$, $s > \frac{5}{9}$.

The high-low frequency technique introduced in \cite{B20} is not strong enough to obtain global solution to the full range of Sobolev indices for which local solution exists. Recently, Co\-lliander, Keel, Staffilani, Takaoka and Tao ( see \cite{CKSTT, CKSTT-2}) introduced the so called $I$ operator method and almost conserved quantities to obtain global well-posedness of the Cauchy problem where  no conserved quantities are available. This method has been very successful to get sharp global result for several dispersive models, \cite{CKSTT-3, CKSTT-4, Farah} are just a few to mention.

The authors in \cite{CKSTT-2} used the so called $I$ operator method and almost conserved quantities to obtain sharp  global well-posedness results for the KdV and mKdV equations in real line as well as in periodic setting. We note that, the KdV and mKdV equations are related by the Miura's transform, which is a nonlinear mapping and involves derivative. In \cite{CKSTT-2}, the authors first obtained sharp global solution in $H^s(\R)$, $s> -\frac34$, for the KdV equation and then used Miura's transform to prove global well-posedness for the mKdV equation \eqref{ivpmkdv} for data in $H^s(\mathbb{R})$, $s> \frac14$.

\subsection{Main Results}

In this work, our objective is two fold: first we want to prove that the IVP \eqref{ivpmkdv} is globally well-posed  for data in $H^s(\mathbb{R})$, $s> \frac14$, without using the Miura's transform, then use this technique to get the similar global well-posedness result for the mKdV system
\eqref{ivp-sy}.

As far as we know, there is no Miura's transform available to treat the global well-posedness for the system \eqref{ivp-sy}. So it is necessary to develop a method that addresses this global well-posedness issue directly.

In what follows we state the main results of this work. Our first main result is concerned with the global well-posedness to the mKdV equation (\ref{ivpmkdv}) and reads as follows:
\begin{theorem}\label{mkdvglob}
For any $\phi \in H^s(\R),\,s>\frac14$, the unique local solution
 to the IVP (\ref{ivpmkdv}) provided by Theorem \ref{loca} extends to any
time interval $[0, T]$.
\end{theorem}

The second main result, about de mKdV system (\ref{ivp-sy}), is  the following:
\begin{theorem}\label{cmkdvglob}
For any $(\phi, \psi) \in H^s(\R)\times H^s(\R),\,s>\frac14$, the unique local solution
 to the IVP (\ref{ivp-sy}) given by Theorem \ref{loc-sys} extends to any
time interval $[0, T]$.
\end{theorem}

\begin{remark}
Our first main result, though reproduces the result proved in \cite{CKSTT-2} it has the merit of being independent, because the result in \cite{CKSTT-2} depends on the sharp global result for the KdV equation and the Miura's transform. The second result is totally new and improves the result in \cite{MP1}, where the author proved global well-posedness in $H^s(\R)\times H^s(\R)$, $s> \frac49$.
\end{remark}

To prove the main results stated above, we use the second generation of the modified energy and almost conserved quantities introduced by
Colliander, Keel, Staffilani, Takaoka and Tao in \cite{CKSTT, CKSTT-2}. In this method, the Fourier transform restriction norm space $X_{s,b}$ (see \eqref{xsb} below) plays a vital role.  The best local well-posedness result to the IVP \eqref{ivpmkdv} and \eqref{ivp-sy}, proved in \cite{KPV3} and \cite{Monte}, respectively, use the smoothing effect of Kato type combined with the maximal function estimate and Leibniz rule for  fractional derivatives. As our work on the global result heavily depends on the local result obtained by the Fourier transform restriction norm method, we will reproduce the following local well-posedness theorems  using this method (see also \cite{Tao} for the mKdV equation).

\begin{theorem}\label{loca}
 Let $s\geq \frac14$, then for any $\phi\in H^s(\R)$, there exist $\delta = \delta(\|\phi\|_{H^s})$ (with $\delta(\rho)\to \infty$ as $\rho\to 0$) and a unique solution $u$ to the IVP \eqref{ivpmkdv} in the time interval $[0, \delta]$. Moreover, the solution satisfies the estimate
\begin{equation}\label{loc-1}
\|u\|_{X^{\delta}_{s, b}}\lesssim \|\phi\|_{H^s},
\end{equation}
where the norm $\|u\|_{X^{\delta}_{s, b}}$ is as defined in \eqref{xsb-rest}.

\end{theorem}

\begin{theorem}\label{loc-sys}
Let $s\geq \frac14$, then for any $(\phi,\psi)\in H^s(\R)\times H^s(\R)$, there exist $\delta = \delta(\|(\phi,\psi)\|_{H^s\times H^s})$ (with $\delta(\rho)\to \infty$ as $\rho\to 0$) and a unique solution $(u,v)$ to the IVP \eqref{ivp-sy} in the time interval $[0, \delta]$. Moreover, the solution satisfies the estimate
\begin{equation}\label{est-1s}
\|(u, v)\|_{X^{\delta}_{s, b}\times X^{\delta}_{s, b}}\lesssim \|(\phi,\psi)\|_{H^s\times H^s}.
\end{equation}
\end{theorem}

\subsection{General Notations}
Before leaving this section, we list some more notations that will be used in this work.
For $f:\R\times [0, T] \to \R$ we define the mixed
 $L_x^pL_T^q$-norm by
\begin{equation*}
\|f\|_{L_x^pL_T^q} = \left(\int_{\R}\left(\int_0^T |f(x, t)|^q\,dt
\right)^{p/q}\,dx\right)^{1/p},
\end{equation*}
with usual modifications when $p = \infty$. We replace $T$ by $t$ if $[0, T]$ is the whole real line $\R$.

We use $\widehat{f}(\xi)$ to denote  the Fourier transform of $f(x)$ defined by
$$
\widehat{f}(\xi) = c \int_{\R}e^{-ix\xi}f(x)dx$$
and
$\widetilde{f}(\xi)$ to denote  the Fourier transform of $f(x,t)$ defined by
$$
\widetilde{f}(\xi, \tau) = c \int_{\R^2}e^{-i(x\xi+t\tau)}f(x,t)dxdt$$

We use $H^s$  to denote the $L^2$-based Sobolev space of order $s$ with norm
$$\|f\|_{H^s(\R)} = \|\langle \xi\rangle^s \widehat{f}\|_{L^2_{\xi}},$$
where $\langle \xi\rangle = 1+|\xi|$.

Next, we introduce the  Fourier transform norm spaces, more commonly known as Bourgain's space in our analysis.

For $s, b \in \R$,  we define the Fourier transform restriction norm space $X_{s,b}(\R\times\R)$ with norm
\begin{equation}\label{xsb}
\|f\|_{ X_{s, b}} = \|(1+D_t)^b U(t)f\|_{L^{2}_{t}(H^{s}_{x})} = \|\langle\tau-\xi^3\rangle^b\langle \xi\rangle^s \widetilde{f}\|_{L^2_{\xi,\tau}},
\end{equation}
 where $U(t) = e^{-t\partial^{3}_{x}}$ is the unitary group associated
 with the linear problem.

If $b> \frac12$, the Sobolev lemma imply that, $ X_{s, b} \subset C(\R ; H^s_x(\R)).$ For any interval $I$, we define the localized spaces $X_{s,b}^I:= X_{s,b}(\R\times I)$ with norm
\begin{equation}\label{xsb-rest}
\|f\|_{ X_{s, b}(\R\times I)} = \inf\big\{\|g\|_{X_{s, b}};\; g |_{\R\times I} = f\big\}.
\end{equation}
Sometimes we use the definition $X_{s,b}^{\delta}:=\|f\|_{ X_{s, b}(\R\times [0, \delta])}$.

 We use $c$ to denote various  constants whose exact values are immaterial and may
 vary from one line to the next. We use $A\lesssim B$ to denote an estimate
of the form $A\leq cB$ and $A\sim B$ if $A\leq cB$ and $B\leq cA$. Also, we
use the notation $a+$ to denote $a+\epsilon$ for $0< \epsilon \ll 1$.

\secao{Local well-posedness results}\label{linear1}

In this section we provide a proof of Theorems \ref{loca} and \ref{loc-sys} using $X_{s,b}$ spaces.  This method  was used in \cite{Tao} to reproduce the local well-posedness for the mKdV equation. In order to apply the $I$-method and almost conserved quantity to get global result,  we need  a variant of local well-posedness result based on it, so a sketch of proof is presented here.

We define a cut-off function $\psi_1 \in C^{\infty}(\R;\; \R^+)$ which is even, such that $0\leq \psi_1\leq 1$ and
\begin{equation}\label{cut-1}
\psi_1(t) = \begin{cases} 1, \quad |t|\leq 1,\\
                          0, \quad |t|\geq 2.
            \end{cases}
\end{equation}
We also define $\psi_T(t) = \psi_1(t/T)$, for $0\leq T\leq 1$.

In what follows we list some estimates that are crucial in the proof of local result.
\begin{lemma}\label{lemma1}
For any $s, b \in \R$, we have
\begin{equation}\label{lin.1}
\|\psi_1U(t)\phi\|_{X_{s,b}}\leq C \|\phi\|_{H^s}.
\end{equation}
Further, if  $-\frac12<b'\leq 0\leq b<b'+1$ and $0\leq \delta\leq 1$, then
\begin{equation}\label{nlin.1}
\|\psi_{\delta}\int_0^tU(t-t')f(u(t'))dt'\|_{X_{s,b}}\lesssim \delta^{1-b-b'}\|f(u)\|_{X_{s, b'}}.
\end{equation}
Moreover, for $b>\frac12$ and $s\geq \frac14$ we have
\begin{equation}\label{tlin}
\|(u^3)_x\|_{X_{s,b'}}\lesssim  \|u\|_{X_{s,b}}^3\quad( \text{nonlinear estimate}).
\end{equation}
\end{lemma}
\begin{proof}
For the proof of \eqref{lin.1} and \eqref{nlin.1} we refer to \cite{GTV} and for \eqref{tlin} to \cite{Tao}.
\end{proof}

Now we are in position to sketch  proofs of Theorems \ref{loca} and \ref{loc-sys}.

\begin{proof}[Proof of Theorem \ref{loca}] To obtain the local solution, one can use the cut-off functions in the Duhamel's formula as,
\begin{equation}\label{Duhamel}
u(t) = \psi_1(t)U(t)\phi -\psi_{\delta}(t)\int_0^tU(t-t')(u^3)_x(t')dt',
\end{equation}
where $0\leq\delta\leq 1$.

Consider $M>0$ and define a ball $\mathcal{B}$ in the space $X_{s,b}$ by
$$\mathcal{B} := \bigl\{v\in X_{s,b};\; \|v\|_{X_{s,b}}\leq M\bigl\}.$$

For $\theta := 1-b-b' >0$, if we choose $M = 2c\|\phi\|_{H^s}$ and $\delta \lesssim \|\phi\|_{H^s}^{-\frac2\theta}$, then by using \eqref{lin.1}, \eqref{nlin.1} and \eqref{tlin} it can be shown that the mapping
\begin{equation}\label{Duhamel-2}
\Psi u(t) = \psi_1(t)U(t)\phi -\psi_{\delta}(t)\int_0^tU(t-t')(u^3)_x(t')dt',
\end{equation}
is a contraction on $\mathcal{B}$. So, by the standard fixed point argument we can conclude the local well-posedness of the IVP \eqref{ivpmkdv} for given data in $H^s$, $s\geq \frac14$. Moreover,
\begin{equation*}
\|u\|_{X_{s,b}^{\delta}} \leq \|\phi\|_{H^s},
\end{equation*}
where $[0, \delta] $ is the local existence time interval.
\end{proof}


\begin{proof}[Proof of Theorem \ref{loc-sys}]
We define spaces $Z_{s,b}:=X_{s, b}\times X_{s, b}$ and $Y^s:= H^s\times H^s$ with norms  $\|(u, v)\|_{X_{s, b}\times X_{s, b}} := \max \{\|u\|_{X_{s,b}}, \|v\|_{X_{s,b}}\}$ and similar for $Y^s$.  Let $a>0$ and consider a ball in $Z_{s, b}$ given by
\begin{equation}\label{sy-ball}
 \mathcal{X}_{a}^s = \bigl\{ (u, v) \in Z_{s, b};\:\|(u, v)\|_{Z_{s, b}} < a\bigl\}.
 \end{equation}

As we are interested in finding local solution to the IVP \eqref{ivp-sy}, we define the following application with the use cut-off functions
\begin{equation}\label{contrac-s2}
\begin{cases}
\Phi_{\phi}[u, v](t):= \psi_1(t)U(t)\phi - \psi_{\delta}(t)\displaystyle\int_{0}^{t}U(t-t')\partial_x(uv^2)(t')\,dt',\\
\Psi_{\psi}[u, v](t):= \psi_1(t)U(t)\psi - \psi_{\delta}(t)\displaystyle\int_{0}^{t}U(t-t')\partial_x(u^2v)(t')\,dt'.
\end{cases}
\end{equation}

 We will show that, there exist $a>0$ and $\delta >0$ such that the application $\Phi\times \Psi$ maps
 $\mathcal{X}_{a}$ into $\mathcal{X}_{a}$ and is a contraction.

Exploiting the symmetry of the system, we will estimate only the first
component $\Phi$. The estimates for the second component $\Psi$ are similar.

Using \eqref{lin.1}, \eqref{nlin.1} and \eqref{tlin}, we have for $s\geq \frac14$ and $\theta := 1-b-b' >0$,
\begin{equation}\label{sy-1}
\begin{split}
\|\Phi\|_{X_{s,b}}&\leq C\|\phi\|_{H^s} +C\delta^{\theta}\|(uv^2)_x\|_{X_s,b}\\
 &\leq C\|\phi\|_{H^s} +C\delta^{\theta}\|u\|_{X_{s,b}}\|v\|_{X_{s,b}}^2\\
& \leq C\|(\phi, \psi)\|_{Y^s} +C\delta^{\theta}\|(u, v)\|_{Z_{s,b}}^3.
\end{split}
\end{equation}

Similarly
\begin{equation}\label{sy-2}
\|\Psi\|_{X_{s,b}}\leq C\|(\phi, \psi)\|_{Y^s} +C\delta^{\theta}\|(u, v)\|_{Z_{s,b}}^3.
\end{equation}

Therefore, from \eqref{sy-1} and \eqref{sy-2}, we obtain
\begin{equation}\label{sy-3}
\|(\Phi, \Psi)\|_{Z_{s,b}}\leq C\|(\phi, \psi)\|_{Y^s} +C\delta^{\theta}\|(u, v)\|_{Z_{s,b}}^3.
\end{equation}

Let us choose $a = 2C\|(\phi, \psi)\|_{Y^s}$,  then from \eqref{sy-3}, we get
\begin{equation}\label{sy-4}
\|(\Phi, \Psi)\|_{Z_{s,b}}\leq \frac a2 + C\delta^{\theta} a^3.
\end{equation}

Now, if we take $C\delta^{\theta} a^2\leq \frac12$, i.e., $\delta \lesssim \|(\phi, \psi)\|_{Y^s}^{-\frac2\theta}$, then \eqref{sy-4} yields
\begin{equation}\label{sy-5}
\|(\Phi, \Psi)\|_{Z_{s,b}}\leq \frac a2 + \frac a2 = a.
\end{equation}

Therefore, the application $\Phi\times\Psi$ maps $\mathcal{X}_{a}^s$ into $\mathcal{X}_{a}^s$. With the similar technique, one can easily show that $\Phi\times\Psi$ is a contraction. Hence by a standard argument one can prove that the IVP \eqref{ivp-sy} is locally well-posed for initial data $(\phi, \psi)\in Y^s$ for any $s\geq \frac14$. Moreover, from \eqref{sy-5} and the choice of $a$, one has
\begin{equation}\label{sy-6}
\|(\Phi, \Psi)\|_{Z_{s,b}}\leq C\|(\phi, \psi)\|_{Y^s}
\end{equation}
 and the proof is finished.
 \end{proof}

\secao{Modified energy functional}\label{sec3}

Before introducing modified energy functional, we define $n$-multiplier and $n$-linear functional.

Let $n\geq 2$ be an even integer. An $n$-multiplier $M_n(\xi_1, \dots, \xi_n)$ is a function defined on the hyper-plane $\Gamma_n:= \{(\xi_1, \dots, \xi_n);\;\xi_1+\dots +\xi_n =0\}$ with Dirac delta $\delta(\xi_1+\cdots +\xi_n)$ as a measure.

If $M_n$ is an $n$-multiplier and $f_1, \dots, f_n$ are functions on $\R$, we define an $n$-linear functional, as
\begin{equation}\label{n-linear}
\Lambda_n(M_n;\; f_1, \dots, f_n):= \int_{\Gamma_n}M_n(\xi_1, \dots, \xi_n)\prod_{j=1}^{n}\hat{f_j}(\xi_j).
\end{equation}
We write $\Lambda_n(M_n):=\Lambda_n(M_n;\; f,f,\dots,f)$ in the case when $\Lambda_n$ is applied to the $n$ copies of the same function $f$.

Using Plancherel identity, the energy $E(t)$ defined in \eqref{con.2} can be written in terms of the $n$-linear functional as
\begin{equation}\label{e.2}
E(t)= -\frac12\Lambda_2(\xi_1\xi_2)-\frac1{12}\Lambda_4(1).
\end{equation}

In what follows we record a lemma that relates the time-derivative of the $n$-linear functional defined for the solution $u$ of the mKdV equation.
\begin{lemma}\label{derivative}
Let $u$ be a solution of the IVP \eqref{ivpmkdv} and $M_n$ be a symmetric $n$-multiplier, then
\begin{equation}\label{der.1}
\frac{d}{dt}\Lambda_n(M_n) = \Lambda_n(M_n\alpha_n)-in\Lambda_{n+2}(M_n(\xi_1, \dots, \xi_{n-1}, \xi_n+\xi_{n+1}+\xi_{n+2})(\xi_n+\xi_{n+1}+\xi_{n+2})),
\end{equation}
where $\alpha_n = i(\xi_1^3+\cdots +\xi_n^3)$.
\end{lemma}

Given $s<1$ and a parameter $N\gg 1$, we define a multiplier operator
\begin{equation}\label{I-1}
\widehat{If}(\xi) = m(\xi)\hat{f}(\xi),
\end{equation}
where
\begin{equation}\label{m-2}
m(\xi)=\begin{cases} 1, \quad \qquad\quad|\xi|\leq N,\\
                        \big(\frac{N}{|\xi|}\big)^{1-s}, \quad\; |\xi|\geq 2N,
          \end{cases}
          \end{equation}
is a smooth, radially symmetric and nonincreasing.

 Note that $I$ is a smoothing operator of order $1-s$, in fact
\begin{equation}\label{sm-1}
\|u\|_{X_{s_0, b_0}}\leq c \|Iu\|_{X_{s_0+1-s, b_0}}\leq cN^{1-s}\|u\|_{X_{s_0, b_0}}.
\end{equation}

Now we introduce the first modified energy
\begin{equation}\label{mod-1}
E^1(u):= E(Iu).
\end{equation}

Using Plancherel identity, we can write the first modified energy in terms of the $n$-linear functional as
\begin{equation}\label{mod-12}
E^1(u)= -\frac12\Lambda_2(m_1\xi_1m_2\xi_2)-\frac1{12}\Lambda_4(m_1m_2m_3m_4),
\end{equation}
where $m_j =m(\xi_j)$.

We define the second generation of the modified energy as
\begin{equation}\label{sec-m1}
E^2(u):= -\frac12\Lambda_2(m_1\xi_1m_2\xi_2)-\frac1{12}\Lambda_4(M_4(\xi_1, \xi_2, \xi_3, \xi_4)),
\end{equation}
where the multiplier $M_4$ is to be chosen later.

Now using the identity \eqref{der.1}, symmetrizing and using the fact that $m$ is even, we get
\begin{equation}\label{sec-m2}
\begin{split}
\frac{d}{dt} E^2(u) &= i\Lambda_4\Big((m_1^2\xi_1^3+\cdots +m_4^2\xi_4^3)-M_4(\xi_1, \dots, \xi_4)(\xi_1^3+\xi_2^3+\xi_3^3+\xi_4^3)\Big)\\ \quad &+\frac{i}3\Lambda_6(M_4(\xi_1, \xi_2, \xi_3, \xi_4+\xi_5+\xi_6)(\xi_4+\xi_5+\xi_6))
\end{split}
\end{equation}

If we choose,
\begin{equation}\label{m.4}
M_4(\xi_1, \xi_2, \xi_3, \xi_4) = \frac{m_1^2\xi_1^3+\cdots +m_4^2\xi_4^3}{\xi_1^3+\xi_2^3+\xi_3^3+\xi_4^3},
\end{equation}
then we get $\Lambda_4 =0$.

So, for this choice of $M_4$, we have
\begin{equation}\label{second-m3}
\frac{d}{dt} E^2(u) =\frac{i}3\Lambda_6\Big(M_4(\xi_1, \xi_2, \xi_3, \xi_4+\xi_5+\xi_6)(\xi_4+\xi_5+\xi_6)\Big)=:\Lambda_6(M_6).
\end{equation}

In what follows, we consider $M_4$ given by \eqref{m.4} and $M_6$ defined by
\begin{equation}\label{m.6}
M_6 = M_4(\xi_1, \xi_2, \xi_3, \xi_{456})\xi_{456} =\frac{m_1^2\xi_1^3+m_2^2\xi_2^3+m_3^2\xi_3^3+m^2(\xi_{456})\xi_{456}^3}{\xi_1^3+\xi_2^3+\xi_3^3+\xi_{456}^3} \xi_{456},
\end{equation}
where we have used the notation $\xi_{ijk} = \xi_i+\xi_j+\xi_k$. We recall that on $\Lambda_n$ ($n=4,6$), one has $\xi_1+\cdots+\xi_n =0$.

\secao{Pointwise Multilinear Bounds}\label{sec4}
This section  is devoted to the analysis of the multipliers $M_4$ and $M_6$  introduced in the previous section.
The estimates obtained will be applied in the proof of the almost conservation property in the next section.

\subsection{Notations and preliminary calculus} Before stating our main nonlinear estimates, we recall some important
notation introduced in \cite{CKSTT-2}. Let $\xi=(\xi_1, \ldots, \xi_n)$ be the vector of frequencies such that
\begin{equation}\label{n-Hyperplane}
\Gamma_n = \bigl\{\xi=(\xi_1, \ldots, \xi_n)\in \mathbb{R}^n;\; \xi_1+\xi_2+ \cdots + \xi_n=0 \bigl\}.
\end{equation}
In our context we are interested in the cases $n=4$ and $n=6$. We define $N_i:=|\xi_i|$ and  $N_{ij}:=|\xi_{ij}|$, where $\xi_{ij}=\xi_i+\xi_j$, and we denote by
$$1\le h_1, h_2, h_3, h_4 \le n,\; (n=4,6)$$
the distinct indices such that
$$N_{h_1}\ge N_{h_2}\ge N_{h_3}\ge N_{h_4}$$
are the highest, second highest, third highest, and fourth highest values of the frequencies
$N_1, N_2,\ldots,N_n$, respectively.

We recall the following arithmetic fact:
\begin{equation}\label{Arithmetic-Fact}
\xi_1+\xi_2+ \xi_3+\xi_4=0\Longrightarrow \xi_1^3+\xi_2^3+ \xi_3^3+\xi_4^3=-3(\xi_1+\xi_2)(\xi_1+\xi_3)(\xi_2+\xi_3).
\end{equation}

Furthermore, the following \emph{Double Mean Value Theorem} (DMVT) will be useful to obtain the main estimates.

\begin{lemma}[DMVT]\label{DMVT}
Assume that $f\in C^2(\mathbb{R})$ and that $\max \{|\lambda|, |\eta|\} \ll |\xi|$. Then,
$$|f(\xi+ \lambda+ \eta)-f(\xi +\lambda )-f(\xi+\eta) + f(\xi)|\lesssim |f''(\xi_{\theta})||\lambda||\eta|,$$
where $|\xi_{\theta}|\sim |\xi|$.
\end{lemma}

\subsection{Multiplier bounds} Now we give the  main estimates for the multipliers $M_4$ and $M_6$ defined in
(\ref{m.4}) and (\ref{m.6}), respectively.

\begin{lemma}[Multiplier Bounds]\label{Mn-Est-Lemma}
The  multipliers $M_4$ and $M_6$  satisfy the following estimates:
\begin{align}
&|M_4(\xi_1,\xi_2,\ldots,\xi_4)| \lesssim m^2(N_{h_1})\label{M4-Est-Lemma-1}\\
\intertext{and}
&|M_6(\xi_1,\xi_2,\ldots,\xi_6)| \lesssim m^2(N_{h_1})N_{h_1},\label{M6-Est-Lemma-1}
\end{align}
where $m(\xi)$ is  the function defined in (\ref{m-2}).
\end{lemma}
\begin{proof}
First, using  (\ref{n-Hyperplane}) and (\ref{Arithmetic-Fact}), we rewrite the expressions for $M_4$ and $M_6$ as follows:
\begin{align}
&M_4(\xi_1, \ldots, \xi_4)=-\frac{\sigma(\xi_1, \xi_2, \xi_3)}{3\xi_{12}\xi_{13}\xi_{23}}\label{m.4-Alternative}\\
\intertext{and}
&M_6(\xi_1,\ldots,\xi_6)=\frac{\sigma(\xi_1,\xi_2, \xi_3)}{3\xi_{12}\xi_{13}\xi_{23}}\xi_{123},\label{m.6-Alternative}
\end{align}
where
\begin{equation}\label{Sigma-Numerator}
\sigma(\xi_1,\xi_2,\xi_3)=m_1^2\xi_1^3+m_2^2\xi_2^3+m_3^2\xi_3^3 - m^2(\xi_{123})\xi_{123}^3.
\end{equation}
By symmetry we can suppose  that $|\xi_{12}| \le |\xi_{13}| \le |\xi_{23}|$ and also without loss of
generality we may assume  that $|\xi_3|\le |\xi_2|\le |\xi_1|$ .

First, we estimate the function $\sigma(\xi_1, \xi_2, \xi_3)$ and for this purpose we separate the analysis into two cases.

\noindent{\textbf{\emph{Case A:}} $\boldsymbol{|\xi_1|\lesssim |\xi_{23}|}$.}
We define the even function $f(\xi)=m^2(\xi)\xi^2$. Note  that
$$f'(\xi) \sim m^2(\xi)\xi=m^2(|\xi|)\xi$$
and that the function $m^2(\xi)\xi$ is nondecreasing.

Now we write the function $\sigma=\sigma(\xi_1,\xi_2,\xi_3)$ as follows:
\begin{equation}\begin{split}\label{m.6-Alt-Case-A-1}
\sigma&=m_1^2\xi_1^3 + m_2^2\xi_2^3+m_3^2\xi_3^3 - m^2(\xi_{123})\xi_{123}^3\\
&=m_1^2\xi_1^3 -\xi_1m_2^2\xi_2^2 + (\xi_1+\xi_2)m_2^2\xi_2^2+m_3^2\xi_3^3 - m^2(\xi_{123})\xi_{123}^2(\xi_1+\xi_2+\xi_3)\\
&=\xi_1\left[m_1^2\xi_1^2 - m_2^2\xi_2^2\right] + \xi_{12}\left[m_2^2\xi_2^2 - m^2(\xi_{123})\xi_{123}^2\right] +
\xi_3 \left[m_3^2\xi_3^2-  m^2(\xi_{123})\xi_{123}^2\right]\\
&=\xi_{13}\left[f(\xi_1) \!-\! f(-\xi_2)\right] + \xi_{12}\left[f(\xi_2)\!-\!f(\xi_{123})\right]+\xi_3\left[f(\xi_3)\!-\! f(\xi_{123})\!-\!f(\xi_1)+f(\xi_2)\right]\\
&=\sigma_{1} +\sigma_{2} + \sigma_{3},
\end{split}\end{equation}
where
\begin{equation}\label{m.6-Alt-Case-A-2}
\begin{cases}
\sigma_{1}=\xi_{13}\left[f(\xi_1) - f(-\xi_2)\right],\\
\sigma_{2}=\xi_{12}\left[f(\xi_2) - f(\xi_{123})\right],\\
\sigma_{3}=\xi_3 \left[f(\xi_3)- f(\xi_{123})-f(\xi_1)+f(\xi_2)\right].
\end{cases}
\end{equation}
Next, we estimate the right hand of the inequality
\begin{equation}\label{m.6-Alt-Case-A-3}
|\sigma(\xi_1,\xi_2,\xi_3)|\le |\sigma_{1}|+ |\sigma_{2}|+ |\sigma_{3}|.
\end{equation}
Using  the \emph{Mean Value Theorem} we have that
\begin{equation*}\begin{split}
|\sigma_{1}|+ |\sigma_{2}|& \lesssim \Bigl(m^2(|\xi_{\theta_1}|)|\xi_{\theta_1}| + |m^2(|\xi_{\theta_2}|)|\xi_{\theta_2}|\Bigl)|\xi_{13}||\xi_{12}|,
\end{split}\end{equation*}
where $|\xi_{\theta_i}| \lesssim |\xi_1|\; (i=1,2)$. So,
\begin{equation}\label{m.6-Alt-Case-A-4}
|\sigma_{1}|+ |\sigma_{2}|\lesssim m^2(|\xi_1|)|\xi_1|\,|\xi_{13}||\xi_{12}|\lesssim m^2(|\xi_1|)|\xi_{13}||\xi_{12}||\xi_{23}|
\end{equation}

To estimate the remaining term we divide into sub-cases:
\begin{enumerate}
\item [$(A_1)$] $\boldsymbol{|\xi_2| \lesssim  |\xi_{13}|}$. In this situation, using that $|\xi_3|\le  |\xi_2|\lesssim |\xi_{13}|$ and the  \emph{Mean Value Theorem}, we have
\begin{equation*}\begin{split}
|\sigma_{3}|&\le|\xi_3|\bigl(|f(\xi_3)- f(\xi_{123})|+ |f(\xi_2)-f(-\xi_1)|\bigl)\\
&\lesssim |\xi_{13}|\bigl(|m^2(|\xi_{\theta_3}|)|\xi_{\theta_3}|+ |m^2(|\xi_{\theta_1}|)|\xi_{\theta_1}|\bigl)|\xi_{12}|,\\
\end{split}\end{equation*}
where $|\xi_{\theta_i}| \lesssim |\xi_1|\; (i=1,2).$ So,
\begin{equation}\label{m.6-Alt-Case-A-5}
|\sigma_{3}|\lesssim m^2(|\xi_1|)|\xi_1|\,|\xi_{13}||\xi_{12}|\lesssim m^2(|\xi_1|)|\xi_{23}|\,|\xi_{13}||\xi_{12}|.
\end{equation}

\item [$(A_2)$] $\boldsymbol{|\xi_{13}| \ll |\xi_2|}$. Here, since $|\xi_{12}|\le |\xi_{13}|\ll |\xi_2|$ we apply the DMVT to obtain
\begin{equation*}\begin{split}
|\sigma_{3}|&=|\xi_3|\bigl(|f(\xi_2-\xi_{12}+\xi_{13})- f(-\xi_1)- f(\xi_{123}) + f(\xi_2)|\bigl)\\
&\lesssim |\xi_1||f''(\xi_{\theta})||\xi_{12}|\xi_{13}|,\\
\end{split}\end{equation*}
where $|\xi_{\theta}|\sim |\xi_2|$ and  we have used that $|\xi_3|\le |\xi_1|$.
Then, using that $|\xi_{12}| \ll |\xi_2|$ implies $|\xi_2|\sim |\xi_1|$  we have
$|f''(\xi_{\theta})|\lesssim m^2(|\xi_1|)$ and consequently it follows that
\begin{equation}\begin{split}\label{m.6-Alt-Case-A-6}
|\sigma_{3}|\lesssim |\xi_1|m^2(|\xi_1|)|\xi_{12}||\xi_{13}|\lesssim m^2(|\xi_1|)|\xi_{12}||\xi_{13}||\xi_{23}|.
\end{split}\end{equation}
\end{enumerate}
Now we combine (\ref{m.6-Alt-Case-A-4}), (\ref{m.6-Alt-Case-A-5}) and (\ref{m.6-Alt-Case-A-6}) to get
\begin{equation}\label{m.6-Alt-Case-A-7}
|\sigma(\xi_1,\xi_2,\xi_3)|\le |\sigma_{1}|+ |\sigma_{2}|+ |\sigma_{3}|
\lesssim m^2(|\xi_1|)\,|\xi_{12}|\xi_{13}||\xi_{23}|.
\end{equation}

\noindent{\textbf{\emph{Case B:}} $\boldsymbol{|\xi_1| \gg  |\xi_{23}|}$.} In this case we write $\sigma$ in the following
form:
\begin{equation}\label{m.6-Alt-Case-B-1}
\sigma=f(\xi_1)\xi_1+f(\xi_2)\xi_2+f(\xi_3)\xi_3-f(\xi_{123})\xi_{123}=\tilde{\sigma}_{1}+\tilde{\sigma}_{2}+\tilde{\sigma}_{3},
\end{equation}
where
\begin{equation}\label{m.6-Alt-Case-B-2}
\tilde{\sigma}_{i}=f(\xi_i)\xi_i-\xi_i f(\xi_{123}), \quad (i=1,2,3).
\end{equation}
On the other hand, we can rewrite the $\tilde{\sigma}_{i}$'s ($i=1,2,3$) terms as follows:
\begin{align}
&\tilde{\sigma}_{2}= \xi_2\Bigl[\underbrace{f(\xi_1-\xi_{13}+\xi_{23})-f(-\xi_3)-f(\xi_{123}) + f(\xi_1)}_{a(\xi_1,\xi_2,\xi_3)}\Bigl] +\, \xi_2f(-\xi_3)-\xi_2f(\xi_1),\label{m.6-Alt-Case-B-3}\\
&\tilde{\sigma}_{3}= \xi_3\bigl[\underbrace{f(\xi_2-\xi_{12}+\xi_{13})-f(-\xi_1)-f(\xi_{123}) + f(\xi_2)}_{b(\xi_1,\xi_2,\xi_3)}\bigl] +\, \xi_3f(-\xi_1)-\xi_3f(\xi_2).\label{m.6-Alt-Case-B-4}\\
&\tilde{\sigma}_{3}= \xi_3\bigl[\underbrace{f(\xi_1-\xi_{12}+\xi_{23})-f(-\xi_2)-f(\xi_{123}) + f(\xi_1)}_{c(\xi_1,\xi_2,\xi_3)}\bigl] +\, \xi_3f(-\xi_2)-\xi_3f(\xi_1).\label{m.6-Alt-Case-B-5}
\end{align}
Now, we arrange (\ref{m.6-Alt-Case-B-3}),  (\ref{m.6-Alt-Case-B-4}) and (\ref{m.6-Alt-Case-B-5}) to get
\begin{align}
&\tilde{\sigma}_{2}=\xi_{12} a(\xi_1,\xi_2,\xi_3)-\xi_1f(\xi_2)+\xi_1f(-\xi_3)-\tilde{\sigma}_{1}+\xi_2f(-\xi_3)-\xi_2f(\xi_1),\label{m.6-Alt-Case-B-6}\\
&\tilde{\sigma}_{3}=\xi_{23} b(\xi_1,\xi_2,\xi_3)-\xi_2f(\xi_3)+\xi_2f(-\xi_1)-\tilde{\sigma}_{2}+\xi_3f(-\xi_1)-\xi_3f(\xi_2).\label{m.6-Alt-Case-B-7}\\
&\tilde{\sigma}_{3}=\xi_{13} c(\xi_1,\xi_2,\xi_3)-\xi_1f(\xi_3)+\xi_1f(-\xi_2)-\tilde{\sigma}_{1}+\xi_3f(-\xi_2)-\xi_3f(\xi_1).\label{m.6-Alt-Case-B-8}
\end{align}
Then, adding the identities (\ref{m.6-Alt-Case-B-6}),  (\ref{m.6-Alt-Case-B-7}) and (\ref{m.6-Alt-Case-B-8}) we have
$$\sigma=\tilde{\sigma}_{1}+\tilde{\sigma}_{2}+\tilde{\sigma}_{3}=\frac{1}{2}
\Bigl[\xi_{12}\,a(\xi_1,\xi_2,\xi_3)+ \xi_{23}\,b(\xi_1,\xi_2,\xi_3) + \xi_{13}\,c(\xi_1,\xi_2,\xi_3)\Bigl],$$
where we used the fact that $f(\xi)$ is an even function. Observing that $|\xi_{12}| \ll |\xi_1|$ implies $|\xi_{1}|\sim |\xi_2|$ and  applying the DMVT to the terms $a$ $b$ and $c$ we obtain
\begin{equation*}\label{m.6-Alt-Case-B-9}
|\sigma|\lesssim \bigl( |f''(\xi_{\theta_a})|+ |f''(\xi_{\theta_b})| + |f''(\xi_{\theta_c})|\bigl)|\xi_{12}||\xi_{13}||\xi_{23}|,
\end{equation*}
where $|\xi_{\theta_{j}}|\sim |\xi_1|$ with $j=a,b,c$. Also, we have
$|f''(\xi_{\theta_j})|\lesssim m^2(|\xi_1|)$,\, ($j=a,b,c$), and hence

\begin{equation}\label{m.6-Alt-Case-B-10}
|\sigma(\xi_1,\xi_2,\xi_3)|\le |\sigma_{1}|+ |\sigma_{2}|+ |\sigma_{3}|
\lesssim m^2(|\xi_1|)\,|\xi_{12}|\xi_{13}||\xi_{23}|.
\end{equation}

Now, inserting the estimates obtained in (\ref{m.6-Alt-Case-A-7}) and (\ref{m.6-Alt-Case-B-10}) in
(\ref{m.4-Alternative}) and (\ref{m.6-Alternative}) we obtain
\begin{align}
&|M_4(\xi_1, \xi_2,\dots,\xi_4)|\lesssim m^2(|\xi_1|)\le m^2(N_{h_1})\\
\intertext{and}
&|M_6(\xi_1, \xi_2,\dots,\xi_6)|\lesssim m^2(|\xi_1|)|\xi_{123}|\lesssim m^2(|\xi_1|)|\xi_1|\le m^2(N_{h_1})N_{h_1}.
\end{align}
Thus, we finished the proof.
\end{proof}


\secao{Almost Conserved Quantity}\label{sec5}

In this section we estimate the growth of the functional
\begin{equation}\label{Functional-Varphi}
\varphi(u_1, u_2, \dots, u_6)=\int_0^{\delta}\Lambda_6(M_6; u_1, \dots, u_6)dt
\end{equation}
that will be used in the proof of the global result in the  next section.
\subsection{Notations and preliminary results} We use the following notation:
$$\mathcal{H}_k=\bigl\{h_1, h_2, \dots, h_k;\; N_{h_1}\ge N_{h_2}\ge \cdots \ge N_{h_k}\bigl\}$$
is the set of the indices of the $k$  dominant frequencies.

The following property will be useful.
\begin{lemma}\label{ACQ-Lemma-Pre-1} Let $n\ge 2$ be an integer and $f_1,\dots, f_n\in \mathcal{S}(\mathbb{R})$ real-valued functions, then we have
\begin{equation*}
\int_{\Gamma_n}\widehat{f}_1(\xi_1)\cdots\widehat{f}_n(\xi_n)dS_{\xi}=\int_{\mathbb{R}}f_1(x)\cdots f_n(x)dx.
\end{equation*}
\end{lemma}
\begin{proof} The identity is valid for  $n=2$. Indeed, by the Plancherel equality we get
\begin{equation*}\begin{split}
\int_{\Gamma_2}\widehat{f}_1(\xi_1)\widehat{f}_2(\xi_2)dS_{\xi}&:=\int_{\mathbb{R}}\widehat{f}_1(\xi_1)\widehat{f}_2(-\xi_1)d\xi_1\\
&=\int_{\mathbb{R}}\widehat{f}_1(\xi_1)\widehat{\bar{f}}_2(\xi_1)d\xi_1
=\int_{\mathbb{R}}f_1(x)f_2(x)dx.
\end{split}\end{equation*}
We proceed by the induction principle. Hence, we suppose that the identity holds for $n-1$ and
then prove it for $n$. Now, using induction argument
\begin{equation*}\begin{split}
\int_{\Gamma_n}\widehat{f}_1(\xi_1)\cdots\widehat{f}_n(\xi_n)dS_{\xi}
&:=\int_{\mathbb{R}^{n-1}}\widehat{f}_1(\xi_1)\cdots\widehat{f}_{n-1}(\xi_{n-1})\widehat{f}_n(-\xi_1\cdots -\xi_{n-1})d\xi_1\dots d\xi_{n-1}\\
&=\int_{\mathbb{R}^{n-2}}\widehat{f}_1(\xi_1)\cdots \widehat{f_{n-1}f_n}(-\xi_1\cdots -\xi_{n-2})d\xi_1\dots d\xi_{n-2}\\
&=\int_{\mathbb{R}}f_1(x)\cdots f_{n-1}(x)f_n(x)dx,
\end{split}\end{equation*}
and this completes the proof.
\end{proof}

Also, we shall take advantage of the following  Strichartz estimates for the Airy group.
\begin{lemma}\label{Strichart-Estimate-Airy}Let $s\ge \frac14$,  $q\in [2, +\infty]$ and $p$ satisfying $\frac{3}{p}=\frac{1}{2}-\frac{1}{q}$. Then
\begin{enumerate}
\item [(a)] $\|f\|_{L^p_tL^q_x}\lesssim \|f\|_{X_{0,\frac{1}{2}+}}$ (Strichartz type estimates),
\item [(b)] $\|f\|_{L^6_tL^6_x}\lesssim \|f\|_{X_{0,\frac{1}{2}+}}$,
\item [(c)] $\|f\|_{L^4_xL^{\infty}_{\delta}}\lesssim \|f\|_{X^{\delta}_{s,\frac{1}{2}+}}$.
\end{enumerate}
\begin{proof}
For the proof of (a) in the case $p=r=8$  we indicate  the reference \cite{KPV5} and for a complete discussion see Lemma 2 in \cite{Grunrock2005A}.
On the other hand, the inequality in (b) can be obtained by interpolation between
$\|f\|_{L^8_{t,x}}\lesssim \|f\|_{X_{0,\frac{1}{2}+}}$
and the trivial estimate
$\|f\|_{L^2_{t,x}}\lesssim \|f\|_{X_{0,0}}$. Finally, we observe that (c) follows from the linear estimate
$\|U(t)\phi\|_{L^4_xL^{\infty}_{\delta}}\le C \|\phi\|_{H^{1/4}}$ and for the complete arguments we refer, for example, the works
\cite{CKSTT-4, Ginibre}.
\end{proof}
\end{lemma}

\begin{lemma}\label{Extra-Estimate} Let $s>1/4$ and $v_1, v_2 \in \mathcal{S}(\mathbb{R})\times \mathcal{S}(\mathbb{R})$ such that
$\emph{supp}\; \widehat{v}_1 \subset \bigl\{\xi;\; |\xi|\sim N \bigl\}$ and
$\emph{supp}\; \widehat{v}_2 \subset \bigl\{\xi;\; |\xi|\ll N \bigl\}$. Then
$$\|v_1v_2\|_{L^4_xL^2_t}\lesssim N^{-1}\frac{1}{4s-1}\|v_1\|_{X_{0,\frac{1}{2}+}}\|v_2\|_{X_{s,\frac{1}{2}+}}.$$
\end{lemma}
\begin{proof} This result was proved in \cite{Carvajal}.
\end{proof}

\subsection{Multi-linear estimate} Now we prove a $6$-linear estimate for the functional $\varphi$ introduced in (\ref{Functional-Varphi})
in terms of the localized Bourgain space  $X^{\delta}_{1,\frac{1}{2}+}$ with the norm defined in \eqref{xsb-rest}.

\begin{proposition}\label{ACQ-Prop-Functional-Grow}Let $u_1(x,t),\dots,u_6(x,t)\in \mathcal{S}(\mathbb{R}\times \mathbb{R})$, then we have
\begin{equation}\label{ACQ-Prop-Functional-Grow-M6}
\left|\int_0^{\delta}\Lambda_6(M_6; u_1, \dots, u_6)dt\right|\lesssim N^{-3}\prod_{j=1}^{6}\|Iu_j\|_{X^{\delta}_{1, \frac{1}{2}+}}.
\end{equation}
\end{proposition}

\begin{proof}
Similar to the technique employed in \cite{CKSTT-2, CKSTT-3, CKSTT-4} we assume that $\widehat{u_j}$ are nonnegative functions and we perform
a  Littlewood-Paley decomposition, where each $\widehat{u_j}$ is restricted to a dyadic frequency ($|\xi_j|\sim N_j$).

Next, we employ the arguments analogous to those used in \cite{Carvajal}. We divide the analysis in the following cases:

\noindent{\textbf{\emph{Case A:}} $\boldsymbol{N \lesssim  N_{h_4}}$.} In this case, using the definition of the function $m(\xi)$, we have
\begin{equation*}\label{Grow-M6-1}
m(N_{h_j})N_{h_j}=\left(\frac{N}{N_{h_j}}\right)^{1-s} N_{h_j}= N^{1-s}N_{h_j}^s\gtrsim N\quad \text{for}\quad j=1,\dots,4.
\end{equation*}
Consequently, using the estimate \eqref{M6-Est-Lemma-1} from Lemma \ref{Mn-Est-Lemma}, one gets
\begin{equation}\label{Grow-M6-2}
|M_6(\xi_1, \dots,\xi_6)|\lesssim N_{h_1}m^2(N_{h_1})\le N^{-3}N_{h_1}m(N_{h_1})N^3\lesssim N^{-3}\prod_{j=1}^{4}N_{h_j}m(N_{h_j}).
\end{equation}

Now, from (\ref{Grow-M6-2}), using Lemma \ref{ACQ-Lemma-Pre-1}, H\"older inequality Lemma \ref{Strichart-Estimate-Airy}-(b) and the regularizing property of the operator $I$ from \eqref{sm-1}, we obtain
\begin{equation}\begin{split}\label{Grow-M6-3}
\left|\int_0^{\delta}\Lambda_6(M_6; u_1, \dots, u_6)dt\right|&\lesssim \int_0^{\delta}\int_{\Gamma_6}|M_6(\xi_1, \dots, \xi_6)|\prod_{j=1}^{6}\hat{u}(\xi_j)\;dS_{\xi}dt\\
&\lesssim N^{-3}\int_0^{\delta}\int_{\mathbb{R}}\prod\limits_{j\in \mathcal{H}_4}D_xIu_j
\prod\limits_{j\not \in \mathcal{H}_4}u_j\,dx dt\\
&\lesssim N^{-3}\prod\limits_{j\in \mathcal{H}_4}\|D_xIu_j\|_{L^6_{\delta} L^6_x}\prod\limits_{j\not \in \mathcal{H}_4}\|u_j\|_{L^6_{\delta} L^6_x}\\
&\lesssim N^{-3}\prod\limits_{j\in \mathcal{H}_4}\|Iu_j\|_{X^{\delta}_{1, \frac{1}{2}+}}
\prod\limits_{j\not \in \mathcal{H}_4}\|Iu_j\|_{X^{\delta}_{1-s, \frac{1}{2}+}}\\
&\lesssim N^{-3}\prod\limits_{j=1}^{6}\|Iu_j\|_{X^{\delta}_{1, \frac{1}{2}+}}.
\end{split}\end{equation}

\noindent{\textbf{\emph{Case B:}} $\boldsymbol{ N_{h_4} \ll N}$.} First we consider  the higher frequency $N_{h_1}\ll N$. In this subcase,  $m(\xi)=1$, and consequently we have $\dfrac{d}{dt}E^2(u)=0$. Hence,  from  (\ref{second-m3}) we get $\Lambda_6(M_6)=0$ and (\ref{ACQ-Prop-Functional-Grow-M6})
is trivial in this subcase.

Next we analyze the complementary subcase: $N_{h_1} \gtrsim N$. Since $\xi_{h_1}+\cdots +\xi_{h_6}=0$ with
$$|\xi_{h_1}|\ge|\xi_{h_2}|\ge \dots \ge|\xi_{h_6}|$$
implies  that
$$|\xi_{h_2}|\le |\xi_{h_1}|=|-\xi_{h_2}\cdots -\xi_{h_6}|\le 5|\xi_{h_2}|$$
we conclude that $N_{h_1}\thicksim N_{h_2} \gtrsim N$ and consequently $N\lesssim m(N_{h_2})N_{h_2}$. Hence, from Lemma \ref{Mn-Est-Lemma}, we have
\begin{equation}\label{Grow-M6-4}
|M_6(\xi_1, \dots,\xi_6)|\lesssim N_{h_1}m^2(N_{h_1})\le N^{-1}N_{h_1}m(N_{h_1})N\lesssim N^{-1}\prod_{j=1}^{2}N_{h_j}m(N_{h_j}).
\end{equation}
Now, using (\ref{Grow-M6-4}) and applying  Lemma \ref{ACQ-Lemma-Pre-1}, Fubini theorem and  H\"older inequality, Lemma \ref{Extra-Estimate}
and Lemma \ref{Strichart-Estimate-Airy}-(c) we obtain

\begin{equation*}\begin{split}\label{Grow-M6-5}
\left|\int_0^{\delta}\Lambda_6(M_6; u_1, \dots, u_6)dt\right|&\lesssim \int_0^{\delta}\int_{\Gamma_6}|M_6(\xi_1, \dots, \xi_6)|\prod_{j=1}^{6}\hat{u}(\xi_j)\;dS_{\xi}dt\\
&\lesssim N^{-1}\int_{\mathbb{R}}\int_0^{\delta}\prod\limits_{j\in \mathcal{H}_2}D_xIu_j\prod\limits_{j\not \in \mathcal{H}_2}u_j\, dt dx\\
&\lesssim N^{-1} \|(D_xIu_{h_1})u_{h_4}\|_{L^4_xL^2_{\delta}}\|(D_xIu_{h_2})u_{h_5}\|_{L^4_xL^2_{\delta}}
\|u_{h_3}u_{h_6}\|_{L^2_xL^{\infty}_{\delta}}\\
&\lesssim N^{-3} \prod\limits_{j=1}^{2}\|Iu_{h_j}\|_{X^{\delta}_{1,\frac{1}{2}+}}\prod\limits_{j=4}^{5}\|u_{h_j}\|_{X^{\delta}_{\frac{1}{4}+,\frac{1}{2}+}}
\prod\limits_{j\in\{3,6\}}\|u_{h_j}\|_{L^4_xL^{\infty}_{\delta}}\\
&\lesssim N^{-3}\prod\limits_{j=1}^{2}\|Iu_{h_j}\|_{X^{\delta}_{1,\frac{1}{2}+}}
\prod\limits_{j=4}^{5}\|Iu_{h_j}\|_{X^{\delta}_{1-s+\frac{1}{4}+,\frac{1}{2}+}}
\prod\limits_{j\in\{3,6\}}\|u_{h_j}\|_{X^{\delta}_{\frac{1}{4},\frac{1}{2}+}}\\
&\lesssim N^{-3}\prod\limits_{j=1}^{6}\|Iu_j\|_{X^{\delta}_{1, \frac{1}{2}+}},
\end{split}\end{equation*}
for all $\frac 14< s < 1$.
\end{proof}

\begin{remark}\label{rmk-1}
With the use of the estimate for $M_4$ in \eqref{M4-Est-Lemma-1} and Lemma \ref{ACQ-Lemma-Pre-1}, it is easy to obtain
\begin{equation}\label{m4-1}
|\Lambda_4(M_4;\; u_1, \dots, u_4)|\lesssim \prod\limits_{j=1}^{4}\|Iu_j\|_{H^1}.
\end{equation}
\end{remark}

We finish this section with the following theorem which says that the second modified energy $E^2(u(t))$ is an almost conserved quantity.
\begin{proposition}\label{almst-cons}
Let $\delta >0$ be given. Then the second generation of the modified energy $E^2(u(t))$ defined in \eqref{sec-m1} satisfies
\begin{equation}\label{alms-2}
|E^2(u(\delta))| \leq |E^2(\phi)| + CN^{-3}\|Iu\|_{X_{1, \frac12+}^{\delta}}^6.
\end{equation}
\end{proposition}
\begin{proof}
The proof follows by using the identity \eqref{second-m3} and estimate \eqref{ACQ-Prop-Functional-Grow-M6} from Proposition \ref{ACQ-Prop-Functional-Grow}.
\end{proof}

\section{Rescaling and Iteration: global results for the mKdV equation}

\subsection{Variant of the local well-posedness} Now we give the auxiliary local results that will be useful in the proof of
the global results.

\begin{theorem}\label{local-variant}
Let $s\geq \frac14$, then for any $\phi$ such that $I\phi\in H^1$, there exist
$\delta = \delta(\|I\phi\|_{H^1})$ (with $\delta(\rho)\to \infty$ as $\rho\to 0$) and a unique solution to the IVP \eqref{ivpmkdv} in the time interval $[0, \delta]$. Moreover, the solution satisfies the estimate
\begin{equation}\label{variant-2}
\|Iu\|_{X^{\delta}_{1, b}}\lesssim \|I\phi\|_{H^1},
\end{equation}
and the local existence time $\delta$ can be chosen satisfying
\begin{equation}\label{delta-var}
\delta \lesssim \|I\phi\|_{H^1}^{-\frac2\theta},
\end{equation}
where $\theta>0$ is as in the proof of Theorem \ref{loca}.
\end{theorem}
\begin{proof}
 Note that, applying the interpolation lemma (Lemma 12.1 in \cite{CKSTT-5}) to \eqref{tlin} we obtain, under the same assumptions on the parameters $s$, $b$ and $b'$ that
 \begin{equation}\label{tlin-I}
 \|I(u^3)_x\|_{X_{1, b'}}\lesssim \|Iu\|_{X_{1,b}}^3,
 \end{equation}
 where the implicit constant does not depend on the parameter $N$ appearing in the definition of the operator $I$.

 Now, using the trilinear estimate \eqref{tlin-I}, one can complete the proof of this theorem exactly as in the proof of Theorem \ref{loca}. So, we omit the details.
\end{proof}


\subsection{Proof of global result for the mKdV equation}

\begin{proof}[Proof of Theorem \ref{mkdvglob}] Given any $T>0$, we are interested in extending the local solution to the IVP \eqref{ivpmkdv} to the interval $[0, T]$.

To make the analysis a bit easy we use the scaling argument. If $u(x,t)$ solves the IVP \eqref{ivpmkdv} with initial data $\phi(x)$ then for $1<\lambda<\infty$, so does $u^{\lambda}(x,t)$ with initial data $\phi^{\lambda}(x)$; where  $u^{\lambda}(x,t)= \frac1\lambda u(\frac x\lambda, \frac t{\lambda^3})$ and $\phi^{\lambda}(x)=\frac1\lambda\phi(\frac x\lambda)$.

Now, our interest is in extending the solution $u^{\lambda}$ to the bigger time interval $[0, \lambda^3T]$.

Observe that
\begin{equation}\label{g-1}
\|\phi^{\lambda}\|_{H^s}\lesssim \frac1{\lambda^{s+\frac 12}}\|\phi\|_{H^s}.
\end{equation}
From this observation and \eqref{sm-1} we have that
\begin{equation}\label{g-2}
\|I\phi^{\lambda}\|_{H^1}\lesssim N^{1-s}\frac1{\lambda^{s+\frac 12}}\|\phi\|_{H^s}.
\end{equation}

If we choose the parameter $\lambda=\lambda(N)$ suitable, we can make $\|I\phi^{\lambda}\|_{H^1}$ as small as we please. In fact, by choosing
\begin{equation}\label{g-4}
\lambda \sim N^{\frac{2(1-s)}{1+2s}},
\end{equation}
we can make
\begin{equation}\label{g-5}
\|I\phi^{\lambda}\|_{H^1}\leq \epsilon.
\end{equation}

Now, from \eqref{g-5} and \eqref{delta-var}, we can guarantee that the rescaled solution $Iu^{\lambda}$ exists in the time interval $[0, 1]$.

Moreover, for this choice of $\lambda$, from \eqref{sec-m1}, using Plancherel identity, \eqref{m4-1}   and \eqref{g-5}, we have
\begin{equation}\label{g-6}
|E^2(\phi^{\lambda})|\lesssim \|I\phi^{\lambda}\|_{H^1}^2 + \|I\phi^{\lambda}\|_{H^1}^4\leq \epsilon^2+\epsilon^4\lesssim \epsilon^2.
\end{equation}

Using the almost conservation law \eqref{alms-2} for the modified energy, \eqref{variant-2}, \eqref{g-5} and \eqref{g-6},  we obtain
\begin{equation}\label{g-7}
\begin{split}
|E^2(u^{\lambda})(1)|&\lesssim |E^2(\phi^{\lambda})| +N^{-3}\|Iu^{\lambda}\|_{X_{1, \frac12+}^1}^6\\
&\lesssim \epsilon^2+N^{-3}\epsilon^6\\
&\lesssim \epsilon^2+N^{-3}\epsilon^2.
\end{split}
\end{equation}

From \eqref{g-6}, it is clear that we can iterate this process $N^3$ times before doubling the modified energy $|E^2(u^{\lambda})|$. Therefore, by taking $N^3$ times steps if size $O(1)$, we can extend the rescaled solution to the interval $[0, N^3]$. As we are interested in extending the the solution to the interval $[0, \lambda^3T]$, we must select $N=N(T)$ such that $\lambda^3T\leq N^3$. Therefore, with the choice of $\lambda$ in \eqref{g-4}, we must have
\begin{equation}\label{g-8}
TN^{\frac{3-12s}{1+2s}}\leq c.
\end{equation}

Hence, for arbitrary $T>0$ and large $N$, \eqref{g-8} is possible if $s>\frac14$. This completes the proof of the theorem.
\end{proof}

\section{Rescaling and Iteration: global results for th mKdV system }

In this section, we deal with the global well-posedness for the mKdV system \eqref{ivp-sy}.
Here also, we follow the I-method and almost conserved quantities to achieve the goal. Several notations introduced and estimates obtained for the single mKdV equation in the earlier sections will be useful.

We start with the time-derivative rule for an $n$-multiplier for solution $(u,v)$ to the mKdV system \eqref{ivp-sy}.
\begin{lemma}\label{sy-deriv}
Let $(u_i, v_i)$ be the $n$-copies of a solution $(u, v)$ to  the IVP \eqref{ivp-sy} and $M_n$ be a symmetric $n$-multiplier, then
\begin{align}
&\frac{d}{dt}\Lambda_n(M_n; u_1, \dots, u_n) = \Lambda_n(M_n\alpha_n; u_1, \dots, u_n)-i\,n\Lambda_{n+2}(a_n(u_i,v_i)),\label{der-s1} \\ 
&\frac{d}{dt}\Lambda_n(M_n; v_1, \dots, v_n) = \Lambda_n(M_n\alpha_n; v_1, \dots, v_n)-i\,n\Lambda_{n+2}(a_n(v_i,u_i)), \label{der-s2}
\end{align}
where
$a_n(f_i, g_i)=M_n(\xi_1, \dots, \xi_{n-1}, \xi_n+\xi_{n+1}+\xi_{n+2})(\xi_n+\xi_{n+1}+\xi_{n+2}); f_1, \dots, f_n, g_{n+1}, g_{n+2})$
and $\alpha_n = i(\xi_1^3+\cdots +\xi_n^3)$. 
\end{lemma}

Notice that \eqref{der-s1} and \eqref{der-s2} are symmetric on $u$ and $v$. So, we consider any one of them that suits the situation under consideration.

As earlier, we define the  first modified energy
\begin{equation}\label{me-s1}
E^1(u, v):= I_2(Iu, Iv),
\end{equation}
where $I_2$ is as defined in \eqref{con22}.

Using Plancherel identity, we can write the first modified energy in terms of the $n$-linear functional as
\begin{equation}\label{me-s2}
E^1(u, v)= -\Lambda_2(m_1\xi_1m_2\xi_2; u,u)-\Lambda_2(m_1\xi_1m_2\xi_2; v,v) -\Lambda_4(m_1m_2m_3m_4; u,u,v,v).
\end{equation}

We define the second generation of the modified energy as
\begin{equation}\label{me-s3}
E^2(u, v):= -\Lambda_2(m_1\xi_1m_2\xi_2; u,u)-\Lambda_2(m_1\xi_1m_2\xi_2; v,v)-\Lambda_4(\tilde{M}_4(\xi_1, \xi_2, \xi_3, \xi_4); u,u,v,v),
\end{equation}
where the multiplier $\tilde{M}_4$ is to be chosen later.

Now using the identity \eqref{der-s1}, \eqref{der-s2}, symmetrizing and using the fact that $m$ is even, we get
\begin{equation}\label{sec-s21}
\begin{split}
\frac{d}{dt} E^2(u,v) &=-\Lambda_2(m_1\xi_1m_2\xi_2\alpha_2; u,u) + 2i\Lambda_4(m_1\xi_1m(\xi_2+\xi_3+\xi_4)(\xi_2+\xi_3+\xi_4)^2; u,u,v,v)\\
&\quad-\Lambda_2(m_1\xi_1m_2\xi_2\alpha_2; v,v) + 2i\Lambda_4(m_1\xi_1m(\xi_2+\xi_3+\xi_4)(\xi_2+\xi_3+\xi_4)^2; v,v,u,u)\\
&\quad-\Lambda_4(\tilde{M}_4(\xi_1, \dots, \xi_4)\alpha_4; u,u,v,v) \\
&\quad + 4i\Lambda_6(\tilde{M}_4(\xi_1,\xi_2,\xi_3, \xi_4+\xi_5+\xi_6)(\xi_4+\xi_5+\xi_6); u,u,v,v,v,v).
\end{split}
\end{equation}

Note that, on $\Gamma_2$, $\alpha_2 =0$, and on $\Gamma_4$, $\xi_2+\xi_3+\xi_4 =-\xi_1$. Therefore, using the fact that $m$ is even and symmetrizing the multiplier on $\Lambda_4$, we get from \eqref{sec-s21} that
\begin{equation}\label{sec-s22}
\begin{split}
\frac{d}{dt} E^2(u,v) &=  4i\Lambda_4((m_1^2\xi_1^3+\cdots +m_4^2\xi_4^3)-\frac14\tilde{M}_4(\xi_1, \dots, \xi_4)(\xi_1^3+\xi_2^3+\xi_3^3+\xi_4^3));u,u,v,v)\\
& \quad  + 4i\Lambda_6(\tilde{M}_4(\xi_1, \xi_2, \xi_3, \xi_4+\xi_5+\xi_6)(\xi_4+\xi_5+\xi_6); u,u,v,v,v,v).
\end{split}
\end{equation}

If we choose,
\begin{equation}\label{me-s4}
\tilde{M}_4(\xi_1, \xi_2, \xi_3, \xi_4) = 4\frac{m_1^2\xi_1^3+\cdots +m_4^2\xi_4^3}{\xi_1^3+\xi_2^3+\xi_3^3+\xi_4^3},
\end{equation}
then we get $\Lambda_4 =0$.

So, for this choice of $M_4$, we have
\begin{equation}\label{sec-s3}
\frac{d}{dt} E^2(u,v) =4i\Lambda_6(\tilde{M}_4(\xi_1, \xi_2, \xi_3, \xi_4+\xi_5+\xi_6)(\xi_4+\xi_5+\xi_6); u,u,v,v,v,v)=:\Lambda_6(\tilde{M}_6).
\end{equation}

As in the single mKdV case, in what follows, we consider $\tilde{M}_4$ given by \eqref{me-s4} and $M_6$ defined by
\begin{equation}\label{sy-m6}
\tilde{M}_6 = \tilde{M}_4(\xi_1, \xi_2, \xi_3, \xi_{456})\xi_{456} =4\frac{m_1^2\xi_1^3+m_2^2\xi_2^3+m_3^2\xi_3^3+m^2(\xi_{456})\xi_{456}^3}{\xi_1^3+\xi_2^3+\xi_3^3+\xi_{456}^3} \xi_{456}.
\end{equation}

Since $\tilde{M}_4$ and $\tilde{M}_6$ are the constant multiples of $M_4$ and $M_6$ respectively, as in Lemma \ref{Mn-Est-Lemma}, we get the following bounds
\begin{align}
&|\tilde{M}_4(\xi_1,\xi_2,\ldots,\xi_4)| \lesssim m^2(N_{h_1})\label{sy-m4-bd}
\intertext{and}
&|\tilde{M}_6(\xi_1,\xi_2,\ldots,\xi_6)| \lesssim m^2(N_{h_1})N_{h_1}.\label{sy-m6-bd}
\end{align}

Similarly as in Proposition \ref{ACQ-Prop-Functional-Grow} and Remark \ref{rmk-1},
for $u_1(x,t),\dots,u_6(x,t)\in \mathcal{S}(\mathbb{R}\times \mathbb{R})$, one can obtain
\begin{equation}\label{sy-M6.1}
\left|\int_0^{\delta}\Lambda_6(\tilde{M}_6; u_1, \dots, u_6)dt\right|\lesssim N^{-3}\prod_{j=1}^{6}\|Iu_j\|_{X^{\delta}_{1, \frac{1}{2}+}},
\end{equation}
and
\begin{equation}\label{sy-m4-1}
|\Lambda_4(\tilde{M}_4; u_1, \cdots, u_4)|\lesssim \prod\limits_{j=1}^{4}\|Iu_j\|_{H^1}.
\end{equation}

As in the single mKdV equation, for given $\delta >0$, with the use of \eqref{sec-s3} and \eqref{sy-M6.1}, it is easy to obtain the following almost conservation law
\begin{equation}\label{sy-alms}
\begin{split}
|E^2(u,v)(\delta)|&\leq |E^2(\phi,\psi)|+CN^{-3} \|Iu\|_{X^{\delta}_{1, \frac{1}{2}+}}^2\|Iv\|_{X^{\delta}_{1, \frac{1}{2}+}}^4\\
&\leq |E^2(\phi,\psi)|+CN^{-3} \|(Iu, Iv)\|_{Z^{\delta}_{1, \frac{1}{2}+}}^6.
\end{split}
\end{equation}

In what follows, we present a variant of the local well-posedness result for the mKdV system \eqref{ivp-sy} after introducing the smoothing operator $I$.
\begin{theorem}\label{sy-variant}
Let $s\geq \frac14$, then for any $(\phi, \psi)$ such that $(I\phi, I\psi)\in Y^1$, there exist
$\delta = \delta(\|I\phi\|_{H^1}, \|I\psi\|_{H^1})$ (with $\delta(\rho)\to \infty$ as $\rho\to 0$) and a unique solution to the IVP \eqref{ivp-sy} in the time interval $[0, \delta]$. Moreover, the solution satisfies the estimate
\begin{equation}\label{var-sy1}
\|(Iu, Iv)\|_{Z^{\delta}_{1, b}}\lesssim \|(I\phi, I\psi)\|_{Y^1}, \qquad b>\frac12,
\end{equation}
and the local existence time $\delta$ can be chosen satisfying
\begin{equation}\label{del-sy}
\delta \lesssim \|(I\phi, I\psi)\|_{Y^1}^{-\frac2\theta},
\end{equation}
where $\theta>0$ is as in the proof of Theorem \ref{loc-sys}.
\end{theorem}
\begin{proof}
 The proof of this result follows exactly as in Theorem \ref{loc-sys}
 using the trilinear estimate \eqref{tlin-I}. So, we omit the details.
\end{proof}

Now we are in position to supply the proof of the global well-posedness result for the mKdV system \eqref{ivp-sy}.
\begin{proof}[Proof of Theorem \ref{cmkdvglob}]
The idea of proof  is  similar to the one we used to prove Theorem \ref{mkdvglob}. For the sake of clearness we give all details involved in the proof.
Here too, we are interested in extending the local solution to the IVP \eqref{ivp-sy} to the interval $[0, T]$ for any arbitrary $T>0$.

The IVP \eqref{ivp-sy} is invariant under scaling, i.e., if $(u(x,t), v(x,t))$ solves the IVP \eqref{ivp-sy} with initial data $(\phi(x), \psi(x))$ then for $1<\lambda<\infty$, so does $(u^{\lambda}(x,t), v^{\lambda}(x,t))$ with initial data $(\phi^{\lambda}(x), \psi^{\lambda}(x))$; where  $u^{\lambda}(x,t)= \frac1\lambda u(\frac x\lambda, \frac t{\lambda^3})$, $\phi^{\lambda}(x)=\frac1\lambda\phi(\frac x\lambda)$ and similar for $v$ and $\psi$.

With scaling argument introduced above, we need to extend the solution $(u^{\lambda}, v^{\lambda})$ to the bigger time interval $[0, \lambda^3T]$.

Observe that
\begin{equation}\label{sy-g-1}
\|(\phi^{\lambda}, \psi^{\lambda})\|_{Y^s}
\lesssim \frac1{\lambda^{s+\frac 12}}\|(\phi, \psi)\|_{Y^s}.
\end{equation}
From this observation and \eqref{sm-1} we have that
\begin{equation}\label{sy-g-2}
\|(I\phi^{\lambda}, I\psi^{\lambda})\|_{Y^1}\lesssim N^{1-s}\frac1{\lambda^{s+\frac 12}}\|(\phi, \psi)\|_{Y^s}.
\end{equation}

If we choose the parameter $\lambda=\lambda(N)$ suitable, we can make $\|(I\phi^{\lambda}, I\phi^{\lambda})\|_{Y^1}$ as small as we please. In fact, by choosing
\begin{equation}\label{sy-g-4}
\lambda \sim N^{\frac{2(1-s)}{1+2s}},
\end{equation}
we can make
\begin{equation}\label{sy-g-5}
\|(I\phi^{\lambda}, I\phi^{\lambda})\|_{Y^1}\leq \epsilon.
\end{equation}

Now, from \eqref{sy-g-5} and \eqref{del-sy}, we can guarantee that the rescaled solution $Iu^{\lambda}$ exists in the time interval $[0, 1]$.

Moreover, for this choice of $\lambda$, from \eqref{me-s3}, with the use of Plancherel identity, \eqref{sy-m4-1}   and \eqref{sy-g-5}, we have
\begin{equation}\label{sy-g-6}
|E^2(\phi^{\lambda}, \psi^{\lambda})|\lesssim \|(I\phi^{\lambda}, I\psi^{\lambda})\|_{Y^1}^2 + \|(I\phi^{\lambda}, I\psi^{\lambda})\|_{Y^1}^4\leq \epsilon^2+\epsilon^4\lesssim \epsilon^2.
\end{equation}

Using the almost conservation law \eqref{sy-alms} for the modified energy, \eqref{var-sy1}, \eqref{sy-g-5} and \eqref{sy-g-6},  we obtain
\begin{equation}\label{sy-g-7}
\begin{split}
|E^2(u^{\lambda}, v^{\lambda})(1)|&\lesssim |E^2(\phi^{\lambda},\phi^{\lambda} )| +N^{-3}\|(Iu^{\lambda}, Iu^{\lambda})\|_{Z_{1, \frac12+}^1}^6\\
&\lesssim \epsilon^2+N^{-3}\epsilon^6\\
&\lesssim \epsilon^2+N^{-3}\epsilon^2.
\end{split}
\end{equation}

From \eqref{sy-g-6}, it is clear that we can iterate this process $N^3$ times before doubling the modified energy $|E^2(u^{\lambda})|$. Therefore, by taking $N^3$ times steps if size $O(1)$, we can extend the rescaled solution to the interval $[0, N^3]$. As we are interested in extending the the solution to the interval $[0, \lambda^3T]$, we must select $N=N(T)$ such that $\lambda^3T\leq N^3$. Therefore, with the choice of $\lambda$ in \eqref{g-4}, we must have
\begin{equation}\label{sy-g-8}
TN^{\frac{3-12s}{1+2s}}\leq c.
\end{equation}

Hence, for arbitrary $T>0$ and large $N$, \eqref{sy-g-8} is possible if $s>\frac14$. This completes the proof of the theorem.
\end{proof}

\bigskip
{\bf Acknowledgment.} Part of this research was carried while A. J. Corcho was visiting the Center
of Mathematics of University of Minho, Portugal supported by CAPES-Brazil and by the 2010 FCT-CAPES project {\it Nonlinear waves and dispersion}.

\end{document}